\documentclass[12pt,a4paper,twoside,reqno]{amsart}
\usepackage{amsfonts}

\usepackage[OT1]{fontenc}
\usepackage{type1cm}
\usepackage[english]{babel}
\usepackage{amsthm}
\usepackage{epsfig}
\usepackage{graphicx,epsfig,subfigure}

\numberwithin{equation}{section}
\theoremstyle{definition}
\newtheorem{thm}{Theorem}[section]
\theoremstyle{definition}
\newtheorem{lm}[thm]{Lemma}
\theoremstyle{definition}

\theoremstyle{definition}

\theoremstyle{definition}
\newtheorem{df}[thm]{Definition}

\theoremstyle{remark}
\newtheorem{rem}[thm]{Remark}

\newcommand{\Rn}{\mathbb{R}^{n}}

\newcommand{\R}{\mathbb{R}}

\newcommand{\C}{\mathbb{C}}

\newcommand {\grtrsim} {\ {\raise-.5ex\hbox{$\buildrel>\over\sim$}}\ }
\newcommand {\lesssim} {\ {\raise-.5ex\hbox{$\buildrel<\over\sim$}}\ }

\newcommand{\khii}{\text{\lower -.4ex\hbox{$\chi$}}}

\begin{document}
\title[Boundedness of singular integrals]
{Boundedness and convergence for singular integrals of measures separated by
Lipschitz graphs}
\author{Vasilis Chousionis and Pertti Mattila}

\thanks{The first author is supported by the Finnish Graduate School in
Mathematical Analysis.} 

\begin{abstract}We shall consider the truncated singular integral
operators
\begin{equation*}
T_{\mu ,K}^{\varepsilon}f(x)=\int_{\mathbb{R}^{n}\setminus
B(x,\varepsilon )}K(x-y)f(y)d\mu y
\end{equation*}
and related maximal operators
$T_{\mu,K }^{\ast }f(x)=\underset{\varepsilon >0}{\sup }\left| T_{\mu,K
}^{\varepsilon }f(x)\right|$. We shall prove for a large class of kernels
$K$ and  measures $\mu$ and $\nu$ that if $\mu$ and $\nu$ are separated by a
Lipschitz graph, then $T_{\nu,K }^{\ast }:L^p(\nu)\to L^p(\mu)$ is bounded for
$1<p<\infty$. We shall also show that the truncated operators
$T_{\mu ,K}^{\varepsilon}$
converge weakly in some dense subspaces of $L^2(\mu)$ under mild
assumptions for the measures and the kernels.
\end{abstract}

\maketitle

\section{Introduction}

Let $K:\mathbb{R}^{n}\setminus \{0\}\rightarrow \mathbb{R}$ be some
continuously differentiable function and $\mu $ some finite Radon measure in $%
\mathbb{R}^{n}$. The truncated singular integral operators associated with $%
\mu $ and $K$ are given for $f\in L^1(\mu)$ by

\begin{equation*}
T_{\mu ,K}^{\varepsilon}f(x)=\int_{\mathbb{R}^{n}\setminus B(x,\varepsilon
)}K(x-y)f(y)d\mu y .
\end{equation*}
Here $B(x,\varepsilon)$ is the closed ball centered at $x$ with
radius $\varepsilon$. Since the kernels we are interested in will
remain fixed in the proofs, although the measures might vary, we
will use the notation $ T_{\mu}^{\varepsilon}$ instead of $T_{\mu
,K}^{\varepsilon}$. Following this convention, the maximal singular
integral operator is defined as
\begin{equation*}
T_{\mu }^{\ast }f(x)=\underset{\varepsilon >0}{\sup }\left| T_{\mu
}^{\varepsilon }f(x)\right|.
\end{equation*}

One of the key concepts in the theory of singular integral operators
is $L^2$ boundedness. It is well known that even with very nice
kernels the boundedness of $T_{\mu }^{\ast }:L^2(\mu)\to L^2(\mu)$
requires strong regularity properties of $\mu$. In this paper we
consider two measures $\mu$ and $\nu$ which live on different sides
of some $(n-1)$-dimensional Lipschitz graph. We shall prove that
then $T_{\nu }^{\ast }:L^2(\nu)\to L^2(\mu)$ is bounded very
generally. The case where $\nu =\mathcal{H}^{n-1}\lfloor S$, the
restriction of the $(n-1)$-dimensional Hausdorff measure to a
Lipschitz graph $S$, was proved by David in \cite{D} and our proof
relies on this result. We shall apply our boundedness theorem to
show that the truncated operators $T_{\mu }^{\varepsilon}$ converge
weakly in some dense subspaces of $L^2(\mu)$.

Before stating our main results we give some basic definitions that
determine our setting.
\begin{df}
The class $\Delta $ will contain all finite Radon measures $\mu $ on
$\Rn$ such that
\begin{equation}
\mu (B(x,r))\leq C_{\mu }r^{n-1} \text{ for } x\in\Rn \text{ and } r>0,
\end{equation}
where $C_{\mu }$ is some constant depending on $\mu $.
\end{df}

We restrict to finite Radon measures only for convenience. Since by definition
Radon measures are always locally finite, all our results easily extend
to general Radon measures.

\begin{df}
The class $\mathcal{K}$ will contain all continuously differentiable
kernels $K:\mathbb{R}^{n}$ $\backslash \{0\}\rightarrow \mathbb{R}$
satisfying for all $x\in \mathbb{R}^{n}$ $\backslash \{0\}$,

\begin{enumerate}
\item  $K(-x)=-K(x)$ (Antisymmetry),

\item  $| K(x)| \leq C_{0}^K| x| ^{-(n-1)}$,

\item  $\left| \nabla K(x)\right| \leq C_{1}^K\left| x\right| ^{-n}$,
\end{enumerate}
where the constants $C_{0}^K$ and $C_{1}^K$ depend on $K$.
\end{df}
The classes $\mathcal{K}$ and $\Delta$ have been studied widely, see
e.g. \cite{Db} and the references therein. Notice also that both
$\mathcal{K}$ and $\Delta$ are quite broad. For example, the class
$\Delta$ contains measures supported on
$(n-1)$-dimensional planes and Lipschitz graphs but it also contains
measures whose support is some fractal set like the $1$-dimensional
four corners Cantor set in $\R^2$. Moreover Riesz kernels
$|x|^{-n}x, x\in\Rn$, belong to $\mathcal{K}$, as well as stranger kernels like
the ones appearing in \cite{D3}.

Denote the graph of a function $f:\mathbb{R}
^{n-1}\rightarrow \mathbb{R}$ by
\begin{equation*}
C_{f}=\left\{(x,f(x)): x\in \mathbb{R}^{n-1}\right\}
\end{equation*}
and the corresponding half spaces by
\begin{equation*}
H_{f}^{+}=\{(x,y):x\in \mathbb{R}^{n-1},y>f(x)\}\text{ and }
H_{f}^{-}=\{(x,y):x\in \mathbb{R}^{n-1},y<f(x)\}.
\end{equation*}

Our first main result reads as follows.
\begin{thm}
\label{mthm} Let $f:\mathbb{R}^{n-1}\rightarrow \mathbb{R}$ be some
Lipschitz function and $\mu$ and $\nu$  measures in $\Rn$ such that
\begin{enumerate}
\item $\mu(H_{f}^{-})=\nu(H_{f}^{+})=0$,
\item $\mu,\nu \in \Delta$.
\end{enumerate}
There exist constants $C_p,1\leq p  <\infty,$ depending only on
$p,n,C_{\mu}, C_{\nu}$ and $\textmd{Lip}(f)$ such that for all
$g\in L^1(\nu)$,
\begin{equation*}
\label{mainp} \int (T_{\nu }^{\ast }g)^p d\mu\leq
C_p\int|g|^p d\nu\text{ for }1 <p <\infty.
\frac{C_1}{t}\int |g|d\nu \text{ for }t>0.
\end{equation*}
\end{thm}

The proof is based on the following two theorems. The first one is a
special case of a classical result, for related discussion and
references see \cite{DS}, p.13. The second was proved by David in
\cite{D}. Although David worked only in the plane, his proof
generalizes without any essential changes.

\begin{thm}
\label{bl} Let $S\subset \mathbb{R}^{n}$ be some $(n-1)$-dimensional
Lipschitz graph and let $\sigma =\mathcal{H}^{n-1}\lfloor S$. Then
if $K\in \mathcal{K}$ the corresponding maximal operator
\begin{equation*}
T_{\sigma }^{\ast }:L^{p}(\sigma )\rightarrow L^{p}(\sigma )
\end{equation*}
is bounded for $1 <p <\infty$.
\end{thm}

\begin{thm}
\label{dav} Let $K\in \mathcal{K}$ and $\mu,\sigma \in \Delta $.
Suppose that there exists a positive constant $c_{\sigma}$
such that $\sigma(B(x,r))\geq c_{\sigma}r^{n-1}$ for $x$ in the support of
$\sigma$ and for $0<r<1$, and that
\begin{equation*}
T_{\sigma }^{\ast }:L^{p}(\sigma )\rightarrow L^{p}(\sigma )
\end{equation*}
is bounded for $1 <p <\infty$. Then
\begin{equation*}
T_{\sigma }^{\ast }:L^{p}(\sigma )\rightarrow L^{p}(\mu )\text{ and
}T_{\mu }^{\ast }:L^{p}(\mu )\rightarrow L^{p}(\sigma )
\end{equation*}
are also bounded for $1 <p <\infty$.
\end{thm}

\begin{rem} The antisymmetry assumption on Theorem \ref{mthm} is not essential in the following sense.
As it was observed in \cite{D88}, Theorem \ref{dav} holds for all kernels $K:\Rn \times \Rn \setminus \{(x,y):x=y\} \rightarrow \R$
which satisfy the estimates
\begin{equation*}
|K(x,y)| \leq C|x-y|^{-(n-1)}
\end{equation*}
and
\begin{equation*}
|\nabla_x K(x,y)|+|\nabla_y K(x,y)|\leq C|x-y|^{-n}.
\end{equation*}
It is evident from the proof of Theorem \ref{mthm} that it remains
true for any of the aforementioned kernels $K$ whose corresponding
maximal operator $T^*_K$ is bounded on
$L^{2}(\mathcal{H}^{n-1}\lfloor C_f)$. As in Theorem \ref{mthm}
$C_f$ stands for the Lipschitz graph that separates the two measures
$\mu$ and $\nu$.
\end{rem}

We shall apply Theorem \ref{mthm} to obtain certain weak convergence results.
Recently it was shown in \cite{MV} that for general
measures and kernels the $L^2(\mu)$-boundedness of the
operators $T_{\mu ,K}^{\varepsilon}$ forces them to converge weakly
in $L^2(\mu)$. This means that there exists a bounded linear
operator $T_{\mu,K }:L^2(\mu)\rightarrow L^2(\mu)$ such that for all
$f,g \in L^2(\mu)$,
\begin{equation*}
\lim _{\varepsilon \rightarrow 0} \int T_{\mu,K }^{\varepsilon
}(f)gd\mu=\int T_{\mu,K }(f)gd\mu .
\end{equation*}
Motivated by this recent development it is natural to ask if limits
of this type might exist if we remove the very strong
$L^2$-boundedness assumption. But, as it was remarked in \cite{MV},
by the Banach-Steinhaus theorem the converse also holds often; weak
convergence implies $L^2$-boundedness. And $L^2$-boundedness is
known to fail very often, for example, by \cite{MeV} and \cite{L},
if $K$ is the Cauchy kernel, $K(z)=1/z, z\in\C$, and  $\mu$ has
positive and finite 1-upper density, i.e, $$0 <\limsup_{r\rightarrow
0}\frac{\mu(B(x,r))}{r} <\infty \text{ $\mu$ a.e,}$$ and is purely
unrectifiable, that is, $\mu(\Gamma)=0$ for every rectifiable curve
$\Gamma$. Hence we cannot hope for the full weak convergence in
$L^2(\mu)$ in such cases. However, we shall prove that the operators
$T_{\mu ,K}^{\varepsilon}$ converge weakly in a restricted sense,
see Theorem \ref{main}, under some mild assumptions for the measures
and the kernels, including also many purely unrectifiable measures.

For these convergence results we shall also use the following theorem.
It was first proved in \cite{MM} for the Cauchy transform in the plane, and then by
a different method by Verdera in \cite{V}. Verdera's proof easily extends
to the present setting, one can also consult \cite{M}, Section 20.

\begin{thm}
\label{pv} Let $S\subset \mathbb{R}^{n}$ be some $(n-1)$-dimensional
Lipschitz graph. Then if $K\in \mathcal{K}$ and $\nu$ is any finite Radon
measure in $\mathbb{R}^{n}$, the principal values
\begin{equation*}
\underset{\varepsilon \rightarrow 0}{\lim }\underset{\left| x-y\right|
>\varepsilon }{\int }K(x-y)d\nu y
\end{equation*}
exist and are finite for $\mathcal{H}^{n-1}$ almost all $x\in S$.
\end{thm}

Using Theorems \ref{mthm} and \ref{pv} we are able to prove rather easily the following fact.
\begin{thm}
\label{first}
\label{bli}
Let $\mu \in \Delta$ and $K \in \mathcal{K}$. Then for any Lipschitz function $f:\mathbb{R}%
^{n-1}\rightarrow \mathbb{R}$ the finite limit
\begin{equation}
\underset{\varepsilon \rightarrow 0}{\lim }\underset{\left|
x-y\right|
>\varepsilon }{\int_{\Rn\backslash H_{f}^{-}}\int_{H_{f}^{-}}}K(x-y)d\mu
yd\mu x
\end{equation}
exists.
\end{thm}

Theorem \ref{bli} is the main tool used to establish weak
convergence. Consider the following function spaces, which are dense
subsets of $L^2(\mu)$ for $\mu \in \Delta$,
\begin{equation*}
\begin{split}
\mathcal{X}_{Q}(\mathbb{R}^{n})=\{f:\mathbb{R}^{n}\rightarrow
\mathbb{R}, f \text{ is a finite linear combination of
characteristic}\\
\text{ functions of rectangles in } \mathbb{R}^{n}\}
\end{split}
\end{equation*}
and
\begin{equation*}
\begin{split}
\mathcal{X}_{B}(\mathbb{R}^{n})=\{f:\mathbb{R}^{n}\rightarrow
\mathbb{R}, f \text{ is a finite linear combination of
characteristic}\\
\text{ functions of balls in } \mathbb{R}^{n}\}.
\end{split}
\end{equation*}
Rectangles in $\mathcal{X}_{Q}$ need not have their sides parallel
to the axis.
\begin{thm}
\label{main} If $\mu \in \Delta $ and $K\in \mathcal{K}$, the finite limit
\begin{equation*}
\underset{\varepsilon \rightarrow 0}{\lim }\int T_{\mu }^{\varepsilon
}(f)(x)g(x)d\mu x
\end{equation*}
exists for $f,g\in \mathcal{X}_{B}(\mathbb{R}^{n})$ and $f,g\in
\mathcal{X}_{Q}(\mathbb{R}^{n})$.
\end{thm}

Theorem \ref{main} was proved in \cite{C2} for more
general kernels $K$ but under more restrictive porosity conditions
on the measure $\mu$. Further discussions on boundedness and convergence
properties of singular integrals with general measures can be found
for example in \cite{M}, \cite{MV}, \cite{T}, \cite{D3} and \cite{C1}.

Throughout this paper $A\lesssim B$ means $A\leq CB$ for some
constant $C$ depending only on the appropriate structural constants,
that is, the dimension $n$, the exponent $p$, the Lipschitz constants of
the Lipschitz graphs and the regularity constants $C_{\mu}$ of the measures.

We would like to thank the referee for some useful comments.

\section{$L^p(\nu)\rightarrow L^p(\mu)$ boundedness}
In this section we prove Theorems \ref{mthm} and \ref{bli}.
\begin{proof}[Proof of Theorem \ref{mthm}] %By standard Calder\'on-Zygmund
%techniques that still work when $\mu,\nu \in Delta$
Let $C>0$ be some constant such that
\begin{equation*}\mu(B(x,r))\leq Cr^{n-1}\text{ and }\nu(B(x,r))\leq
Cr^{n-1}\text{ for }x\in \Rn \text{ and }r>0.
\end{equation*}
Write $\mu=\mu_1+\mu_2$ and $\nu=\nu_1+\nu_2$ where $\mu_1=\mu
\lfloor C_f$ and $\nu_1=\nu \lfloor C_f$. By standard differentiation
theory of measures, see, e.g., \cite{M}, Section 2,
the measures $\mu_1$ and $\nu_1$ are absolutely continuous
with respect to $\sigma =\mathcal{H}^{n-1}\lfloor C_f$ with bounded
Radon-Nikodym derivatives. Hence there exist Borel functions $h_\mu$ and
$h_\nu$ such that  $0\leq h_\mu\lesssim 1$ and $0\leq h_\nu\lesssim 1$ and that
$$d\mu_1=h_\mu d\sigma \text{ and }d\nu_1=h_\nu d\sigma.$$
By Theorems \ref{bl} and \ref{dav} we have for $g\in L^p(\nu)$,
\begin{equation*}
\begin{split}
\int (T_{\nu_1 }^{\ast }g)^p d\mu_1&=\int (T_{\sigma }^{\ast
}(gh_\nu))^ph_\mu d\sigma \lesssim \int |gh_\nu|^p d\sigma \\
&\lesssim\int |g|^p h_\nu d\sigma=\int |g|^p d\nu_1\leq \int |g|^p
d\nu,
\end{split}
\end{equation*}
\begin{equation*}
\int (T_{\nu_1 }^{\ast }g)^p d\mu_2=\int (T_{\sigma }^{\ast
}(gh_\nu))^p d\mu_2 \lesssim \int |gh_\nu|^p d\sigma \leq \int |g|^p
d\nu,
\end{equation*}
and
\begin{equation*}
\int (T_{\nu_2 }^{\ast }g)^p d\mu_1\lesssim \int (T_{\nu_2 }^{\ast
}(g))^p d\sigma \lesssim \int |g|^p d\nu_2\leq \int |g|^p d\nu.
\end{equation*}
As $T_{\nu }^{\ast }\leq T_{\nu_1 }^{\ast }+T_{\nu_2 }^{\ast }$ we
may thus assume that  $\mu=\mu_2$ and $\nu=\nu_2$, that is,
$\mu(H_f^{-}\cup C_f)=\nu(H_f^{+}\cup
C_f)=0$, and also that $g(x)=0$ for $x \in H_f^{+}\cup C_f$.

Let $L>\textmd{max}\{1,\textmd{Lip}(f)\}$. For $x_0=(u_0,f(u_0))\in C_f$
define the cone
$$\Gamma(x_0)=\{(u,t)\in \Rn: t-f(u_0)>4L|u-u_0| \},$$
and observe that
\begin{equation}
\label{cone} |y-x|\geq \frac{1}{8L}|y-x_0| \text { for }y \in \Gamma(x_0),x\in
H_f^{-}.
\end{equation}

We define the non-tangential maximal function $N(g)$ for any
function $g:\R^n\rightarrow \overline{\R}$ by
\begin{equation*}
N(g)(x)=\sup \{|g(y)|:y\in \Gamma(x)\}.
\end{equation*}
For the maximal function $N(g)$, the following $L^p$ estimate holds.
\begin{lm}
\label{nc} For any $0 <p
<\infty$, and any $\mu$ measurable function $g:\R^n\rightarrow \overline{\R}$,
$$\int |g|^p d\mu \lesssim\int_{C_f} N(g)^p d\mathcal{H}^{n-1}.$$
\end{lm}
This follows from the fact that $\mu$ is a Carleson measure in
$H_f^{+}$, i.e.,
$$\mu(B(x,r))\leq C\mathcal{H}^{n-1}(C_f\cap B(x,r))\text{ for }x\in C_f,r>0.$$
A simple proof is given in \cite{Tor} for the case where
$C_f={\R}^{n-1}$ but the same argument holds for general $C_f$.

\begin{lm}
\label{l2} For any $g \in L^1(\nu)$ and any $x \in  C_f$,
$$N(T_{\nu}^{\ast }g)(x)\leq C_N(T_{\nu}^{\ast }g(x)+M_\nu g(x))$$
where
$$M_\nu g(x)=\sup_{r>0}r^{1-n}\int_{B(x,r)}|g|d\nu$$
and $C_N$ depends only on $n,L$ and $C$.
\end{lm}
\begin{proof} Let $y \in \Gamma(x)$ and $\varepsilon>0$.
We will estimate $|T_\nu^{\varepsilon}g(y)|$ by dividing the
argument to two cases. Let $r=|x-y|$ and assume first that
$\varepsilon <r$. Then
\begin{eqnarray*}
\left| T_{\nu }^{\varepsilon }g(x)-T_{\nu }^{\varepsilon
}g(y)\right|  &=&\left| \int_{\Rn \setminus B(x,\varepsilon
)}K(x-z)g(z)d\nu
z-\int_{\Rn \setminus B(y,\varepsilon )}K(y-z)g(z)d\nu z\right|  \\
& \leq & \int_{\Rn \setminus B(x,2r)}\left| K(x-z)-K(y
-z)\right||g(z)| d\nu z  \\
&&+\int_{H_f^{-}\cap B(x,2r)}|K(y-z)||g(z)|d\nu z \\
&&+\left|\int_{B(x,2r)\backslash B(x,\varepsilon
)}K(x-z)g(z)d\nu z\right|  \\
\end{eqnarray*}
We estimate the first integral by integrating over the annuli
$B(x,2^{i}r)\setminus B(x,2^{i-1}r),i\in \mathbb{N},i\geq 2$. By the
Mean Value Theorem we derive that
\begin{eqnarray*}
\left| K(x-z)-K(y-z)\right|  &\leq &\left| \nabla K(\xi
(z))\right| \left|x-y\right|  \\
&\leq &\frac{C_{1}^K\left| x-y\right| }{\left| \xi (z)\right| ^{n}}
\end{eqnarray*}
where $\xi (z)$ lies in the line segment joining $y-z$ to $x-z$.
Furthermore for $i\in \mathbb{N},i\geq 2,$ and $z\in
B(x,2^{i}r)\backslash B(x,2^{i-1}r)$,
\begin{equation*}
\begin{split}
| \xi (z)|  &\geq | x-z| -| \xi(z)-(x-z)| \\
&\geq |x-z| -|(y-z)-(x-z)|  \\
&\geq 2^{i-2}r.
\end{split}
\end{equation*}
Hence
\begin{equation*}
\begin{split}
\int_{\Rn \setminus B(x,2r)}\left| K(x-z)-K(y-z)\right|&|g(z)| d\nu
z \\
&\leq \sum_{i=2}^{\infty}\int_{B(x,2^{i}r)\backslash
B(x,2^{i-1}r)}\frac{C_{1}^K\left| x-y\right| }{\left| \xi (z)\right|
^{n}}|g(z)|d\nu z \\
&\leq
4^{n}C_{1}^K\sum_{i=2}^{\infty}2^{-i}\frac{1}{(2^ir)^{n-1}}\int_{B(x,2^{i}r)}|g(z)|d\nu
z \\
&\leq 4^{n}C_{1}^K M_\nu g(x).
\end{split}
\end{equation*}
For the second integral, using (\ref{cone}) we estimate,
\begin{equation*}
\begin{split}
\int_{H_f^{-}\cap B(x,2r)}|K(y-z)||g(z)|d\nu z &\leq C_0^K
\int_{H_f^{-}\cap B(x,2r)}|y-z|^{1-n}|g(z)|d\nu z \\
&\leq (16L)^{n-1}C_0^K (2r)^{1-n}\int_{B(x,2r)}|g(z)|d\nu z\\
&\leq (16L)^{n-1}C_0^K M_\nu g(x)
\end{split}
\end{equation*}
Obviously the third integral is bounded by $2T_{\nu}^{\ast }g(x)$.
Therefore,
\begin{equation}
\label{fes} \left|T_{\nu }^{\varepsilon }g(y)\right|\leq
3\left|T_{\nu }^{\ast}g(x)\right|+D_1 M_\nu g(x)
\end{equation}
where $D_1=4^{n}C_{1}^K+(16L)^{n-1}C_0^K$.

Secondly, suppose that $\varepsilon \geq r$. Then
\begin{eqnarray*}
\left| T_{\nu }^{\varepsilon }g(x)-T_{\nu }^{\varepsilon
}g(y)\right|  &=&\left| \int_{\Rn \setminus B(x,\varepsilon
)}K(x-z)g(z)d\nu
z-\int_{\Rn \setminus B(y,\varepsilon )}K(y-z)g(z)d\nu z\right|  \\
& \leq & \int_{\Rn \setminus B(x,2\varepsilon)}\left| K(x-z)-K(y
-z)\right||g(z)| d\nu z  \\
&&+\int_{B(x,2\varepsilon)\setminus B(y,\varepsilon)}|K(y-z)||g(z)|d\nu z \\
&&+\left|\int_{B(x,2\varepsilon)\setminus B(x,\varepsilon)}K(x-z)g(z)d\nu z\right|  \\
\end{eqnarray*}
Exactly as before
$$\int_{\Rn \setminus B(x,2\varepsilon)}\left| K(x-z)-K(y
-z)\right||g(z)| d\nu z\leq 4^{n}C_{1}^K M_\nu g(x),$$
$$\int_{B(x,2\varepsilon)\setminus
B(y,\varepsilon)}|K(y-z)||g(z)|d\nu z\leq 2^{n-1}C_0^K M_\nu g(x)$$
and $$\left|\int_{B(x,2\varepsilon)\setminus
B(x,\varepsilon)}K(x-z)g(z)d\nu z\right| \leq 2T_{\nu}^{\ast
}g(x).$$ Therefore,
\begin{equation}
\label{ses} \left|T_{\nu }^{\varepsilon }g(y)\right|\leq
3\left|T_{\nu }^{\ast}g(x)\right|+D_2 M_\nu g(x)
\end{equation}
where $D_2=4^{n}C_{1}^K+2^{n-1}C_0^K.$ Choosing $C_N=D_1$ and
combining (\ref{fes}) and (\ref{ses}) we complete the proof the
Lemma \ref{l2}.
\end{proof}
We can now proceed and finish the proof of Theorem \ref{mthm}. By
Lemmas \ref{nc} and \ref{l2}, Theorems \ref{bl} and \ref{dav}, the $L^p$-boundedness
of $M_{\nu+\sigma}$ (see, e.g., \cite{M}, Theorem 2.19) and the fact that
$g(x)=0$ for $x\in C_f$,
\begin{equation*}
\begin{split}
\int(T_{\nu}^{\ast }g)^p d\mu &\lesssim \int N(T_{\nu}^{\ast }g)^p
d\sigma \\
&\lesssim \int (T_{\nu}^{\ast }g)^p d\sigma+ \int (M_\nu g)^p
d\sigma\\
&\lesssim \int |g |^p d\nu+ \int (M_{\nu+\sigma} g)^p
d(\nu+\sigma)\\
&\lesssim \int |g |^p d\nu+\int |g |^p d(\nu+\sigma)\\
&=2\int |g |^p d\nu.
\end{split}
\end{equation*}
The proof is finished.
\end{proof}

\begin{proof}[Proof of Theorem \ref{bli}]
Denote $\nu=\mu\lfloor H_f^-$ and $\lambda=\mu\lfloor(H_f^+\cup C_f)$.
By Theorem \ref{mthm}
\begin{equation*}
T_{\nu }^{\ast }:L^{2}(\nu )\rightarrow L^{2}(\lambda )
\end{equation*}
is bounded.  Therefore by H\"{o}lder's inequality

$$\int T_{\nu }^{\ast }(1)d\lambda  \leq
\| T_{\nu }^{\ast}(1)\| _{L^{2}(\lambda )}\| 1\| _{L^{2}(\lambda )} \lesssim
\| T_{\nu }^{\ast }(1)\| _{L^{2}(\lambda )} \lesssim\| 1\| _{L^{2}(\nu )}<\infty.$$

For $z\in H_{f}^{+}$ the limit
\begin{equation*}
\underset{\varepsilon \rightarrow 0}{\lim }T_{\nu }^{\varepsilon
}(1)(z)
\end{equation*}
exists since $H_{f}^{+}\cap \textmd{spt}\nu=\emptyset $. Furthermore
by Theorem \ref{pv} the above limit also exists for $\mu $ almost
every $z\in C_f$. Thus by the Lebesgue dominated
convergence theorem we derive that the limit
\begin{equation*}
\underset{\varepsilon \rightarrow 0}{\lim }\int_{H_{f}^{+} \cup
C_{f}}T_{\nu }^{\varepsilon }(1)(z)d\mu z=\underset{\varepsilon
\rightarrow 0}{\lim }\underset{\left| x-y\right| >\varepsilon
}{\int_{\Rn\setminus H_{f}^{-}}\int_{H_{f}^{-}}}K(x-y)d\mu yd\mu x
\end{equation*}
exists and is finite, completing the proof of Theorem \ref{bli}.
\end{proof}
\emph{Remark.} As a corollary of Theorem \ref{bli} and Fubini's
theorem  we derive that the limit
\begin{equation*}
\underset{\varepsilon \rightarrow 0}{\lim }\underset{\left|
x-y\right|
>\varepsilon }{\int_{H_{f}^{+}}\int_{\Rn\backslash H_{f}^{+}}}K(x-y)d\mu
yd\mu x
\end{equation*}
exists under the same assumptions as in Theorem \ref{bli}.

\section{Weak Convergence in $\mathcal{X}_{B}(\mathbb{R}^{n})$ and
$\mathcal{X}_{Q}(\mathbb{R}^{n})$}

To prove Theorem \ref{main} let $f,g\in
\mathcal{X}_{Q}(\mathbb{R}^{n})$ or $f,g\in
\mathcal{X}_{B}(\mathbb{R}^{n})$ be such that
\begin{equation*}
f=\overset{l}{\underset{i=1}{\sum }}a_{i}\chi _{Q_{i}}\text{ and }g=\overset{%
m}{\underset{j=1}{\sum }}b_{j}\chi _{P_{j}},
\end{equation*}
where $a_{i},b_{j}\in \mathbb{R}$ and $Q_{i},P_{j}$ are closed balls or
$Q_{i},P_{j}$ are closed rectangles. Then for $\varepsilon >0$,

\begin{equation*}
\int T_{\mu }^{\varepsilon }f(x)g(x)d\mu
x=\overset{m}{\underset{j=1}{\sum }}\overset{l}{\underset{i=1}{\sum
}}b_{j}a_{i}\underset{\left| x-y\right|
>\varepsilon }{\int_{P_{j\text{ }}}\int_{Q_{i}}}K(x-y)d\mu yd\mu x.
\end{equation*}
Therefore it is enough to show that for balls $P,Q$ or rectangles $P,Q$ the limit
\begin{equation*}
\underset{\varepsilon \rightarrow 0}{\lim }\underset{\left|
x-y\right|
>\varepsilon }{\int_{P}\int_{Q}}K(x-y)d\mu yd\mu x
\end{equation*}
exists. But,
\begin{equation*}
\underset{\left| x-y\right| >\varepsilon
}{\int_{P}\int_{Q}}K(x-y)d\mu yd\mu x =I_1+I_2+I_3+I_4,
\end{equation*}
where,
\begin{eqnarray*}
I_1&=&\underset{\left| x-y\right| >\varepsilon }{\int_{P\cap Q}\int_{P\cap Q}}%
K(x-y)d\mu yd\mu x, \\
I_2&=&\underset{\left| x-y\right| >\varepsilon }{\int_{P\backslash
Q}\int_{P\cap Q}}K(x-y)d\mu yd\mu x, \\
I_3&=&\underset{\left| x-y\right| >\varepsilon }{\int_{P\cap
Q}\int_{Q\backslash P}}K(x-y)d\mu yd\mu x, \\
I_4&=&\underset{\left| x-y\right| >\varepsilon }{\int_{P\backslash
Q}\int_{Q\backslash P}}K(x-y)d\mu yd\mu x.
\end{eqnarray*}
By the antisymmetry of $K$, for every $\varepsilon>0$,
\begin{equation*}
I_1=0.
\end{equation*}
Furthermore by Fubini's theorem $I_3$ is essentially the same with
$I_2$, allowing us to treat only $I_2$ and $I_4$. In that direction
notice that for every rectangle, or ball, say $P$, there exist some
collection of rotations of Lipschitz graphs
$\{F_{i}(P)\}_{i=1}^{2n}$, and disjoint Borel sets
$\{A_{i}(P)\}_{i=1}^{2n}$, such that
\begin{eqnarray*}
\Rn \setminus P&=&\cup_{i=1}^{2n} A_{i}(P),\\
P&\subset &H_{F_{i}(P)}^{-}\cup F_{i}(P),\\
A_{i}(P)&\subset &H_{F_{i}(P)}^{+}.
\end{eqnarray*}
\begin{figure}
\centering
\includegraphics[scale = 0.3]{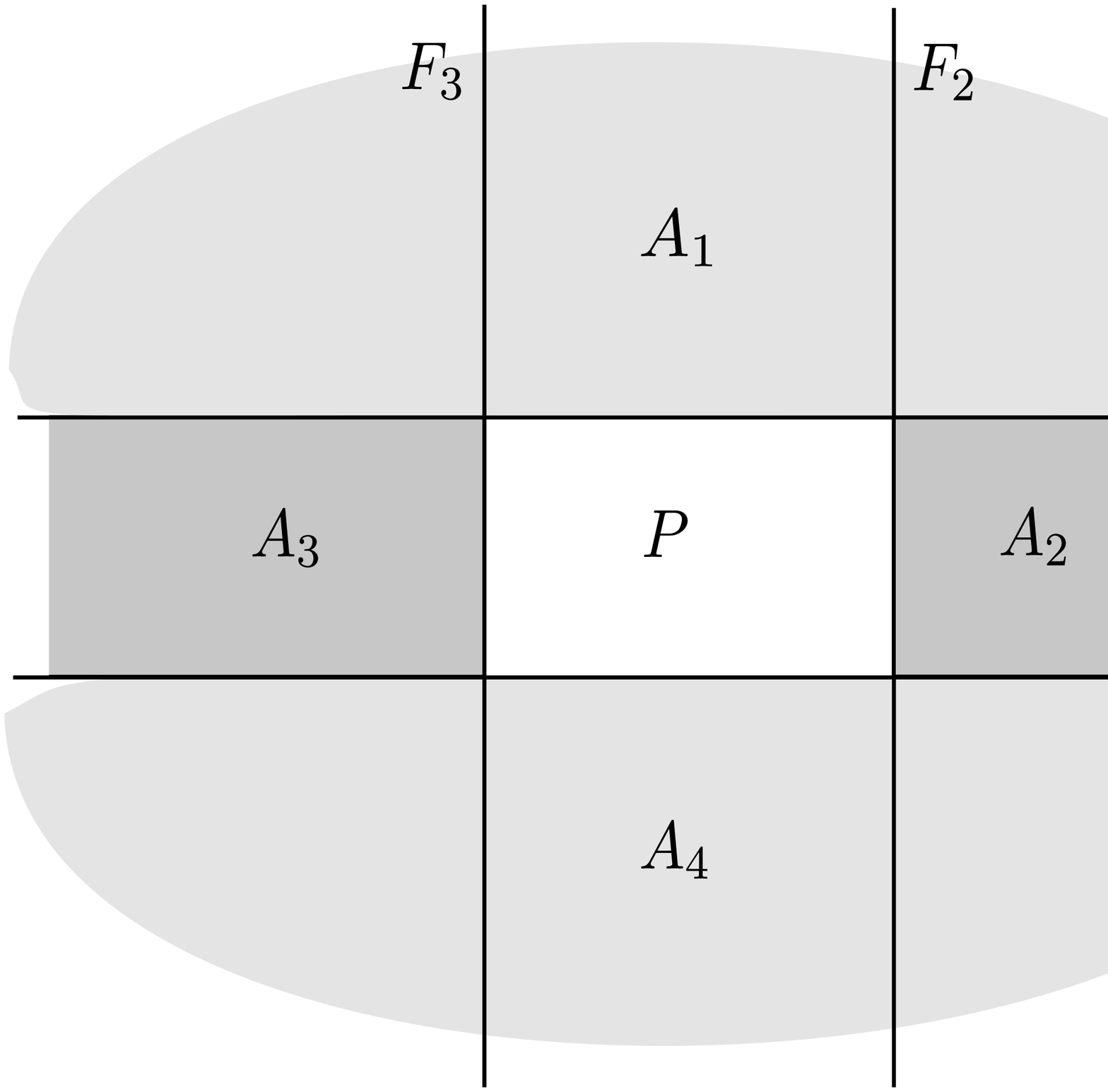}
\caption[]{}\label{fig2}
\end{figure}
See Figure \ref{fig2} for an illustration in the case when $P$ is a
subset of the plane. Using the above geometric property $I_2$ and
$I_4$ can be decomposed in the following way,
\begin{equation*}
I_2=\sum_{i=1}^{2n}\underset{\left| x-y\right|
>\varepsilon }{\int_{A_{i}(Q)\cap P}\int_{P\cap Q}}K(x-y)d\mu yd\mu
x
\end{equation*}
and
\begin{equation*}
I_4=\sum_{i=1}^{2n}\underset{\left| x-y\right|
>\varepsilon }{\int_{A_{i}(Q)\cap P}\int_{Q \setminus P}}K(x-y)d\mu yd\mu
x.
\end{equation*}
Therefore since limits like
\begin{equation*}
\lim_{\varepsilon \rightarrow 0}\underset{\left| x-y\right|
>\varepsilon }{\int_{A_{i}(Q)\cap P}\int_{P\cap Q}}K(x-y)d\mu yd\mu
x
\end{equation*}
and
\begin{equation*}
\lim_{\varepsilon \rightarrow 0} \underset{\left| x-y\right|
>\varepsilon }{\int_{A_{i}(Q)\cap P}\int_{Q \setminus P}}K(x-y)d\mu yd\mu
x
\end{equation*}
exist by Theorem \ref{bli} we finally obtain Theorem \ref{main}.

\vspace{1cm}
\begin{footnotesize}
{\sc Department of Mathematics and Statistics,
P.O. Box 68,  FI-00014 University of Helsinki, Finland,}\\
\emph{E-mail addresses:} \verb"vasileios.chousionis@helsinki.fi",
\verb"pertti.mattila@helsinki.fi"

\end{footnotesize}

\end{document}